\documentclass[12pt,reqno]{amsart}

\usepackage[dvips]{graphicx}

\usepackage%
{xy}
\xyoption{all}

\usepackage{amssymb, latexsym, amsmath, amscd, array, 
}

\newtheorem{theorem}{Theorem}[section]
\newtheorem{lemma}[theorem]{Lemma}
\newtheorem{proposition}[theorem]{Proposition}

\theoremstyle{definition}

\newtheorem{remark}[theorem]{Remark}

\numberwithin{equation}{section}
\numberwithin{figure}{section}
\numberwithin{table}{section}

\newcommand\R {{\mathbb R}}
\newcommand\quat {{\mathbb H}}
\newcommand\Aut{{\rm Aut}}
\newcommand\cay {{\omega^{\phantom{I}}_{\rm Ca}}}

\newcommand\SO {{\rm SO}}
\newcommand\Spin {{\rm Spin}}

\newcommand\pitwo {\pi_2}
\newcommand\pifour {\pi_4}

\title{Cayley form, comass, and triality isomorphisms}
\author[M.~Katz]{Mikhail G. Katz$^{*}$}

\address{Department of Mathematics, Bar Ilan University, Ramat Gan
52900 Israel}

\thanks{$^{*}$Supported by the Israel Science Foundation (grants
no.~84/03 and 1294/06) and the BSF (grant 2006393)}

\author[S.~Shnider]{Steve Shnider}

\email{\{katzmik, shnider\}@macs.biu.ac.il}

\begin{document}

\subjclass[2000]{Primary
53C23;    
Secondary:
17B25
}

\keywords{Cayley form, comass, triality, Wirtinger ratio}

\date{\today}

\begin{abstract}
Following an idea of Dadok, Harvey and Morgan, we apply the triality
property of~$\Spin(8)$ to calculate the comass of self-dual~$4$-forms
on~$\R^8$. In particular, we prove that the Cayley~form has comass~$1$
and that any self-dual~$4$-form realizing the maximal Wirtinger ratio
(equation \eqref{11}) is~$\SO(8)$-conjugate to the Cayley~form. We
also use triality to prove that the stabilizer in~$\SO(8)$ of the
Cayley form is~$\Spin(7)$. The results have applications in systolic
geometry, calibrated geometry, and~$\Spin(7)$ manifolds.
\end{abstract}

\maketitle

\tableofcontents

\section{Introduction}

The Cayley~form, denoted~$\cay$, is a self-dual exterior~$4$-form
on~$\R^8$.  The form~$\cay$ was first defined by R. Harvey and
B. Lawson \cite{HL}, by identifying~$\R^8$ with the Cayley numbers
(octonion algebra) and using well-known constructions of triple and
quadruple vector cross products, see \cite{BG, Cu,Kle}.  We observe
that~$\cay \in \Lambda^4 \R^8$ can be characterized in terms of an
extremal property for the ratio of two norms, the comass norm and the
Euclidean norm on~$\Lambda^4 \R^8$.  Namely,~$\cay$ corresponds to a
point of maximal Euclidean norm in the unit ball of the comass norm
(see Section~\ref{two}).

In systolic geometry \cite{Gr1, Gr2, Gr3, Gr4, SGT}, the Cayley form
plays a key role in the calculation of the optimal stable
middle-dimensional systolic ratio of~$8$-manifolds, and in particular
of the quaternionic projective plane, see~\cite{e7}.  For additional
background on systolic geometry, see \cite{Ka95, BK2, KL, BCIK2, Bru,
DKR}.

The Cayley~form defines an important case in the theory of calibrated
geometries of Harvey and Lawson \cite{HL}.  They remark that ``the
most fascinating and complex geometry discussed here is the geometry
of Cayley~$4$-folds in~${\mathbb R}^8\cong {\mathbb O}$". The
Cayley~form is the calibrating form defining the Cayley~$4$-folds. In
general, a~$k$-form on a Riemannian manifold is called ``calibrating"
if it is closed and has pointwise comass~$1$.

The comass~$\|\omega\|$ of a~$k$-form~$\omega$ on a normed vector
space (such as the tangent space at a point on a Riemannian manifold)
is defined as the maximum of the pairing with
decomposable~$k$-forms~$v_1 \wedge \cdots \wedge v_k$ of norm~$1$:
\begin{equation}
\label{11d}
\| \omega \| = \sup \left\{ \omega(v_1,\ldots,v_k) \left| \; \forall
i, |v_i^{\phantom{I}} |=1 \right. \right\}.
\end{equation}
If~$\phi$ is a calibrating~$k$-form on~${\mathbb R}^n$ with
metric~$g$, a~$k$-dimensional subspace~$\xi$ is said to be calibrated
by~$\phi$ if~$\phi|_\xi= {\rm vol}_{(g|\xi)}$.  A submanifold is said
to be calibrated by a closed calibrating form~$\phi$ if all of its
tangent spaces are calibrated by~$\phi$.  It follows immediately from
the definition and Stokes theorem that a calibrated manifold minimizes
volume within its homology class.

Research on calibrated geometries stimulated by \cite{HL} led to many
new examples of spaces with exceptional holonomy.  For example, the
Cayley~form is the basic building block in the structure of
$8$-manifolds with exceptional~$\Spin(7)$ holonomy, see
\cite{Jo00}. Major contributions in calibrated geometry and
exceptional holonomy have been made by M.~Berger, R. L. Bryant,
D. Joyce, J. Dadok, F. R. Harvey, B.~Lawson, F. Morgan and S. Salamon,
\cite{Ber, Br87, BrHa89, BrSal89, DHM, M88, Sal89, Ha90, Jo96, Jo00,
Jo07}.  Riemannian manifolds with~$G_2$ and~$\Spin(7)$ holonomy, of
dimensions~$7$ and~$8$ respectively, are Ricci flat \cite{Bo}.  The
wealth of new examples of~$\Spin(7)$ and~$G_2$ manifolds constructed
by R.L.~Bryant, D. Joyce, S. Salamon have been used as vacua for
string theories, \cite{Ac, Be, Le, Sha}.  The Cayley~$4$-cycles
on~$\Spin(7)$ manifolds are candidates for the supersymmetric
representatives of fundamental particles \cite{Be}.

A number of authors have calculated the comass~$\| \cay \|$ of the
Cayley~form~$\cay$.  Harvey and Lawson \cite{HL} used a definition of
the Cayley~form in terms of vector cross products of Cayley
numbers. The basic identities they used are derived in a 7 page
appendix. J.~Dadok, R.~Harvey, and F. Morgan \cite{DHM} studied the
self-dual calibrations on~${\mathbb R}^8$ using triality, but their
approach depends on a description of the geometry of polar
representations \cite{Da}.

In this paper, we give an explicit description (for certain weight spaces) 
of the intertwining operator between the triality related representations 
on traceless symmetric~$8\times 8$ matrices(see below) and on self-dual~$4$-forms
on~${\mathbb R}^8$. This allows us to use the representation
of~$\SO(8)$ on traceless symmetric matrices to calculate the comass
and describe the self-dual calibrations without appealing to the
structure theorem for polar representations.

In addition to its relevance for calibrated geometry and special
holonomy, the Cayley form is important for its applications in
systolic geometry.  To help understand the applications, we first
recall the familiar case of~$2$-forms, which is to a certain (but
limited) extent a model for what happens for~$4$-forms.

The space of alternating~$2$-forms on~$\R^{n}$, identified with
antisymmetric matrices on~$\R^n$, becomes a Lie algebra with respect
to the standard bracket~$[A,B]=AB-BA$.  An alternating~$2$-form
$\alpha$ can be decomposed as a sum
\begin{equation}
\label{11c}
\alpha= \sum_i c_i \alpha_i,
\end{equation}
where the summands~$\alpha_i$ are orthonormal, simple and commute
pairwise, i.e. belong to a Cartan subalgebra, see Remark~\ref{12a},
item 2.  Moreover, the summands can be chosen in such a way that the
comass norm~$\|\;\|$ as defined in \eqref{11d}, satisfies
\begin{equation}
\label{11b}
\|\alpha\|= \max_i(|c_i|).
\end{equation}
The standard Euclidean norm on~$\R^n$ extends to a Euclidean
norm~$|\;|$ on all the exterior powers, and we have
\begin{equation}
\label{12}
\frac{|\alpha|^2}{\|\alpha\|^2} \leq {\rm rank},
\end{equation}
where ``rank'' is the dimension of the Cartan subalgebra.  This
optimal bound is attained by the standard symplectic form when~$c_i=1$
for all~$i$.

It turns out that bounds similar to \eqref{12} remain valid
for~$4$-forms on~${\mathbb R}^8$, which are also part of a Lie algebra
structure, defined below, but somewhat surprisingly,
formula~\eqref{11b} is no longer true.  See the counterexample in
Section~\ref{counterx}.  We will prove the following theorem, the
first part of which was proved by different methods in \cite{e7}.

\begin{theorem}
\label{thm}
The Cayley form~$\cay$ has comass~$1$ and satisfies the following
relation:
\begin{equation}
\label{11}
\frac{|\cay^2|}{\|\cay\|^2}=14,
\end{equation}
where the value~$14$ is the maximal possible value for~$4$-forms
on~$\R^8$.
\end{theorem}

The approach using triality also leads to simple proofs of the
following theorems.

\begin{theorem}
\label{thm12} Any self-dual~$4$-form on~$\R^8$ satisfying
\eqref{11} is~$\SO(8)$-conjugate to the Cayley form.
\end{theorem}

\begin{theorem}\label{stabilizer}
The subgroup of~$\SO(8)$ stabilizing the Cayley form is isomorphic
to~$Spin(7)$.
\end{theorem}

\begin{remark}
\label{12a}
\hfil
\begin{enumerate}

\item
In Section~\ref{counterx} we give an example to show that a linear
combination of the seven forms with all coefficients equal
to~$+1$, has comass~$2$, which shows that the situation for
self-dual~$4$-forms on~$\R^8$ is not completely parallel to the
case of~$2$-forms, see equation \eqref{11b}.

\item
 In the course of the proof of  Theorem~\ref{thm12}, we prove that every
self-dual~$4$-form on~$\R^8$ is~$\SO(8)$-conjugate to a linear
combination of the following~$7$ mutually orthogonal self-dual
forms:
\begin{equation}\label{7forms}
\{e^{1234}, e^{1256}, e^{1278}, e^{1357}, e^{1467}, e^{1368},
e^{1458} \},
\end{equation}
where~$e^{jklm}:=e_j\wedge e_k\wedge e_l\wedge e_m + e_p\wedge
e_q\wedge e_r\wedge e_s$ where the second summand is the Hodge
dual of the first. The comment following \eqref{11c} concerning a
Cartan subalgebra is relevant here, because the~$7$ forms listed
in \eqref{7forms} in fact form a maximal abelian subalgebra of
real~$E_7$ as defined in \cite[p.~76]{Ad}.  The conjugacy can be
proved using this fact and a standard theorem in Lie theory.

\item Bryant \cite[p. 545]{Br87} observed that~$|\cay|^2=14$, but
did not notice that this gave the maximal value for the norm of a
calibrating~$4$-form. \end{enumerate}
\end{remark}

One possible application is exploiting the~$\R^8$ estimates described
here so as to calculate the optimal stable middle-dimensional systolic
ratio of~$8$-manifolds.  Such an application depends on the existence
of a Joyce manifold with middle-dimensional Betti number~$b_4=1$.
Currently, it is unknown whether such manifolds exist.

\section{The Cayley form}
\label{two}

The Cayley form can be defined by two different coordinate-dependent
constructions.  There is also a coordinate-independent
characterization of its~$\SO(8)$ orbit.

\begin{proposition}
We have the following three equivalent ways of describing the Cayley
form~$\cay$:
\begin{enumerate}
\item The~$\SO(8)$ orbit of~$\cay$ consists of the set of points
of the unit comass ball in~$\Lambda^4(\R^8)$ of maximal Euclidean
norm. \item Under the identification of~$\R^8$ with~${\mathbb
C}^4$, the Cayley form $\cay$ can be expressed as the sum of two
terms, half the square of a standard Kahler form and the real part
of a holomorphic volume form:
\begin{equation}
\label{intrinsic}
\cay= \frac 12 \omega_J^2 + {\rm Re}(\Omega_J).
\end{equation}
\item
Under the identification of~$\R^8$ with ~$\quat \oplus \quat$, and
quaternionic ``vector space" structure given by right multiplication,
the Cayley form is~$\SO(8)$-equivalent to the alternating sum of half
the squares of the three K\"ahler forms associated with the complex
structures given by right multiplication by~$i,j,k$ respectively, see
\cite[Lemma 2.21]{BrHa89}.  If these forms are denoted~$\omega_{J_a}$,
$a=1,2,3$, then~$\cay$ is~$SO(8)$ conjugate to
\begin{equation}
\label{12b}
\eta_2=-\frac12 \omega_{J_1}^2+\frac12 \omega_{J_3}^2-\frac12
\omega_{J_2}^2.
\end{equation}
\end{enumerate}
\end{proposition}

\begin{remark}\label{12c}
The statement of item 1 was suggested to us by Blaine Lawson.
The forms described in items 2 and 3 of the proposition correspond
to two different points for the orbit described in item~1.  Bryant
and Harvey \cite{BrHa89} identify the Cayley form with
the~$\eta_2$ described in item 3. See Proposition~\ref{morecomassone} 
for the notation. The expression on the right side of equation \eqref{12c}
generalizes to~$n$-dimensional quaternionic space for~$n>2$, and
thus to hyper-K\"ahler manifolds. The Cayley form,  denoted by~$\Phi$
in~\cite[p. 120]{HL} and defined using octonions,  is another point in
the same orbit, $\eta_3$ in the notation of Proposition~\ref{morecomassone} below.
The Cayley form is denoted~$\omega_1$ in \cite[p. 14]{DHM},
and~$\Omega$ in~\cite[p.~342]{Jo00}.
\end{remark}

\begin{proof}
The first assertion of the proposition is a consequence of
Theorem~\ref{thm12}.  The proof is given in Section~\ref{five}.

The simplest description of~$\cay$, the one given in item 2, is based
on the standard identification of~$\R^8$ with~${\mathbb C}^4$.

Let~$\{f_j\}$,~$j=1,\ldots 8$, be an orthonormal basis for~${\mathbb
R}^8$ and~$\{e_j\}$ the dual basis. Define a complex structure by
$$
J(f_{2a-1})= f_{2a}, \quad J(f_{2a})=-f_{2a-1}, \quad a=1,2,3,4.
$$
Then
\begin{equation}
\{ e_{2a-1}+ i e_{2a}, \quad a=1,2,3,4\}
\end{equation}
form a basis for the complex linear dual space.  The definition of the
Cayley form, which uses standard constructions from complex differential
geometry, is as follows.  Define the symplectic form
\begin{equation}\label{omega}
\omega_J=\sum_{a=1,\ldots,4} e_{2a-1}\wedge e_{2a}=\frac12 {\rm Im}
\sum_{a=1,\ldots, 4} (e_{2a-1}- ie_{2a})\otimes (e_{2a-1}+ ie_{2a}),
\end{equation}
and the complex~$4$-form
\begin{equation}\label{Omega}\Omega_J=(e_1+i e_2)\wedge(e_3+i e_4)
\wedge(e_5+i e_6)\wedge(e_7+i e_8);
\end{equation}
then we define
$$\cay:= \frac 12 \omega_J^2 + {\rm Re}(\Omega_J).$$

In terms of the dual basis~$\{e_i|i=1,\ldots,8\}$, the form~$\cay$ is
a signed sum of the~$7$ mutually orthogonal self-dual~$4$-forms
\begin{equation*}
\{e^{1234}, e^{1256}, e^{1278}, e^{1357}, e^{1467}, e^{1368}, e^{1458}
\},
\end{equation*}
where
\begin{equation}
\label{selfdual}
e^{jklm}:=e_j\wedge e_k\wedge e_l\wedge e_m + e_p\wedge e_q\wedge
e_r\wedge e_s
\end{equation}
and the second summand is the Hodge star of the first:

\begin{equation}
\label{4b}
\cay:= e^{1234}+e^{1256}+e^{1278}+ e^{1357}-e^{1368}- e^{1458}-
e^{1467} ,
\end{equation}
see also~\eqref{12b}.

On~$\quat\oplus \quat$, there are three K\"ahler forms defined by the
three complex structures given by right multiplication by~$i,j,k$
respectively. They are
\begin{eqnarray*}
\omega_{J_1}&=&e_1\wedge e_2- e_3\wedge e_4+ e_5\wedge e_6- e_7\wedge
e_8,\\ \omega_{J_2}&=&e_1\wedge e_3- e_4\wedge e_2+ e_5\wedge e_7-
e_8\wedge_6, \quad\mbox{\rm and}\\ \omega_{J_3}&=& e_1\wedge e_4-
e_2\wedge e_3+ e_5\wedge e_8- e_6\wedge e_7.
\end{eqnarray*}
A simple calculation shows that
\begin{eqnarray*}
\eta_2&=&e^{1234}-e^{1256}+e^{1278}- e^{1357}-e^{1368}-e^{1467}+
e^{1458}\\ &=&-\frac12 \omega_{J_1}^2-\frac12 \omega_{J_2}^2+\frac12
\omega_{J_3}^2.
\end{eqnarray*}
That~$\eta_2$ is~$SO(8)$ conjugate to~$\cay$ follows from
Proposition~\ref{morecomassone} and Theorem~\ref{thm12}.
\end{proof}

\section{Triality for~$D_4$}
\label{triality}

The Lie group~$\Spin(8,\mathbb R)$ has three~$8$-dimensional
representations.  They are the vector representation,~$V= \R^8$, and
the two spinor representations,~$\Delta_+$ and~$\Delta_-$.  Fix a
maximal torus~$T \subset \Spin(8)$, and a set of simple positive
roots.  Then for any automorphism~$\phi\in {\rm Aut}(\Spin(8))$, the
image~$\phi(T)$ is another maximal torus.  We can compose with a
conjugation~$\sigma_g(x)=gxg^{-1}$ so that~$\sigma_g\circ \phi(T)=T$
and the fundamental chamber is preserved.  In this way, an element of
the outer automorphism group

\renewcommand{\arraystretch}{1.3}
\begin{figure}
\begin{equation*}
\xymatrix@R=14pt@C=20pt{*=0{\bullet} \ar@{-}[rd] & {} & {} \\{} &
  *=0{\phantom{^{\alpha_2}}\bullet^{\alpha_2}_{\phantom{I}}}
  \ar@{-}[r] & *=0{\bullet} \\*=0{\bullet} \ar@{-}[ur] & {} & {} }
\end{equation*}
\caption{\textsf{Dynkin diagram of~$D_4$, see \eqref{21b}}}
\label{21}
\end{figure}
\renewcommand{\arraystretch}{1}

\begin{equation*}
{\rm Out}(\Spin(8))={\rm Aut}(\Spin(8))/{\rm Inn}(\Spin(8))
\end{equation*}
induces an automorphism of the Dynkin diagram~$D_4$ of
Figure~\ref{21}.  This correspondence determines an isomorphism with
the symmetric group on three letters,~${\rm
Out}(\Spin(8))\cong\Sigma_3$, where the group~$\Sigma_3$ permutes the
three edges of the Dynkin diagram, see J.F.~Adams
\cite[pp.~33-36]{Ad}.  We identify~$so(8)$ with~$8\times 8$ skew symmetric
real matrices  and the Cartan subalgebra ~${\mathfrak h} \subset so(8)$
with the block diagonal matrices having four~$2\times 2$ blocks. An
orthogonal basis~$\{t_1,t_2,t_3,t_4\}$ is defined by the condition:
\begin{equation*}
\sum x_i t_i = {\rm diag}(x_1J, \; x_2J, \; x_3J, \; x_4J),
\end{equation*}
where~$J=\begin{pmatrix}0&1\\-1&0\end{pmatrix}$, while~$\{x_i,\,
i=1,2,3,4\}$ are coordinates in~${\mathfrak h}$.

The simple positive roots~$\alpha_i \in {\mathfrak h}^*$ are
\begin{equation}
\label{21b}
\alpha_1=x_1-x_2,\quad\alpha_2=x_2-x_3,\quad\alpha_3=
x_3-x_4,\quad\alpha_4=x_3+x_4,
\end{equation}
where~$\alpha_2$ appears at the center of the diagram of
Figure~\ref{21}.  The fundamental weights~$\lambda_i \in {\mathfrak
h}^*$ are
\begin{equation*}
\begin{aligned}
\lambda_1 &=x_1, \\
\lambda_2 &=x_1+x_2, \\
\lambda_3 &=\frac12(x_1+x_2+x_3-x_4), \\
\lambda_4 & =\frac12(x_1+x_2+x_3+x_4),
\end{aligned}
\end{equation*}
and the corresponding representations are
\begin{equation*}
\rho_1{\rm \ on\ }\Lambda^1(V)=V, \quad\rho_2{\rm \ on\
}\Lambda^2(V),\quad \rho_3{\rm \ on\ }\Delta_-,\quad\rho_4{\rm \ on\
}\Delta_+,
\end{equation*}
respectively.  Let~$=\sigma^2(V)$ be the representation of~$\Spin(8)$ on
the second symmetric power of~$V$, which, by the~$SO(8)$ equivalence
of~$V$ and~$V^*$, is equivalent to the representation by conjugation on the~$8\times 8$ traceless symmetric matrices.  Let~$\sigma^2_0(V)$ be the
subrepresentation on the traceless symmetric matrices, so that one has
a decomposition
\begin{equation*}
\sigma^2(V)\cong{\mathbf 1}\oplus \sigma^2_0(V).
\end{equation*}
The second symmetric power of~$\Delta_+$ decomposes as
\begin{equation*}
\sigma^2(\Delta_+)={\mathbf 1}\oplus \Lambda^4_+(V),
\end{equation*}
where~$\Lambda^4_+(V)$ is the representation of~$\Spin(8)$ on the
self-dual~$4$-forms, see \cite[p.~25, Theorem~4.6]{Ad}.

The representations
\begin{equation*}
\pitwo: \Spin(8)\rightarrow {\rm Aut}(\sigma^2_0(V))
\end{equation*}
to
\begin{equation*}
\pifour: \Spin(8)\rightarrow {\rm Aut}(\Lambda^4_+(V))
\end{equation*}
both factor  through the vector representation,
$$\rho_1:\Spin(8)\rightarrow \SO(8).$$
If~$\hat \pi_2$ and~$\hat\pi_4$
denote the respective~$\SO(8)$ representations
$$\hat\pi_2:\SO(8)\rightarrow \Aut(\sigma^2_0(V))\quad\mbox{\rm and}$$
$$\hat\pi_4:\SO(8)\rightarrow \Aut(\Lambda^4_+(V),\quad\quad$$
then
\begin{equation}\label{factor}\pi_2=\hat\pi_2\circ\rho_1\quad \pi_4=\hat\pi_4\circ\rho_1.
\end{equation}

Let~$\phi$ be the automorphism (preserving the maximal torus and
fundamental chamber) representing the outer automorphism that
interchanges the fundamental weights~$\lambda_1$ and~$\lambda_4$, and
leaves~$\lambda_3$ fixed. Then~$\phi$ transforms the representation~$\pitwo$ to~$\pifour$.
In other words, there is a linear isomorphism~$\psi:\sigma^2_0(V)
\rightarrow \Lambda^4_+(V)$ such that
\begin{equation}\label{0}
\psi(\pitwo(g)w)=\pifour(\phi(g))\psi(w),
\end{equation}
for all~$g\in \Spin(8)$ and~$w\in\sigma^2_0(V)$.
See Figure~\ref{31}.

\begin{figure}
\begin{equation*}
\xymatrix{ \sigma_0^2(V) \ar[rr]^{\pitwo(g)} \ar[d]^{\psi} &&
\sigma_0^2(V) \ar[d]^{\psi}\\ \Lambda^4_+(V) \ar[rr]^{\pifour\circ
\phi (g)} && \Lambda^4_+(V) }
\end{equation*}
\caption{Intertwining of a pair of representations}
\label{31}
\end{figure}

\section{Weight spaces in symmetric matrices and self-dual~$4$-forms}

In this section we describe the map~$\psi$ in terms of the weights
of~$\sigma^2_0(V)$ and~$\Lambda_+^4(V)$. Since we are dealing with
real representations of a compact group, the weight spaces will be
real two dimensional subspaces.

In the complexified representation~$\sigma^2_0(V)$, the vector
$$(e_{2a-1}+ i e_{2a})\otimes (e_{2a-1}+ i e_{2a})\,\,
\mbox{\rm   is a weight vector with weight  } 2i x_a.$$
In the real representation, we call the 2-dimensional real subspace
with basis
$$u_a=e_{2a-1}\otimes e_{2a-1} -e_{2a}\otimes e_{2a},
\quad\mbox{\rm and  }v_a=e_{2a-1}\otimes e_{2a}+e_{2a}\otimes e_{2a-1}$$
a weight space for the weight,~$2x_a$,~$a=1,2,3,4$.

In terms of traceless symmetric~$8\times 8$ matrices
$so(8)$ acting by matrix commutator, the elementary
formulae:
\begin{equation*}
\left[\left(\begin{array}{cc} 0& 1\\-1& 0\end{array}\right), \quad
\left(\begin{array}{cc} 1& 0\\0& -1\end{array}\right)\right]=
-2\left(\begin{array}{cc} 0& 1\\1&0\end{array}\right)
\end{equation*}
and
\begin{equation*}
\left[\left(\begin{array}{cc} 0& 1\\-1& 0\end{array}\right), \quad
\left(\begin{array}{cc} 0& 1\\1& 0\end{array}\right)\right]=
2\left(\begin{array}{cc} 1& 0\\0& -1\end{array}\right),
\end{equation*}
imply
\begin{equation*} [x_1t_1+x_2t_2+x_3t_3+x_4t_4, u_a]=-2x_a \, v_a
\end{equation*}
and
\begin{equation*}
[x_1t_1+x_2t_2+x_3t_3+x_4t_4, v_a]=2x_a \, u_a,
\end{equation*}
where
\begin{equation*}
u=\left(\begin{array}{cc}1&0\\0&- 1
\end{array}\right), \quad
v=\left(\begin{array}{cc}
0& 1\\
1& 0
\end{array}\right)
\end{equation*}
and
\begin{equation}
\label{23}
u_a=\left(\begin{array}{ccc} 0_{2a-2}& 0& 0\\ 0& u & 0\\ 0& 0& 0_{8-2a}
\end{array}\right), \quad
v_a=\left(\begin{array}{ccc}
0_{2a-2}& 0& 0\\
0& v& 0\\
0& 0& 0_{8-2a}
\end{array}\right).
\end{equation}

The proof of the following lemma is a straightforward calculation.
Recall the notation from \eqref{selfdual}.

\begin{lemma}
The real representation~$\Lambda^4_+$ of\/~$\SO(8,\mathbb R)$ has
highest weight
$2\lambda_4=(x_1+x_2+x_3+x_4)$ corresponding to the
matrix~$(x_1+x_2+x_3+x_4)J$ acting on the two dimensional real weight
space with basis~$\{\mu_1, \nu_1\}$, where
\begin{equation}
\label{first}
\begin{aligned}
\mu_1 &=\mbox{\rm Re} \Omega_J = e^{1357}-e^{1467}-e^{1368}-e^{1458}, \\
\nu_1 &=\mbox{\rm Im} \Omega_J=-e^{1468}+e^{1358}+e^{1457}+e^{1367},
\end{aligned}
\end{equation}
where~$\Omega_J$ is defined in \eqref{Omega}
Conjugating two of the complex linear factors ~$e_{2b-1}+i e_{2b}$
and~$e_{2c-1}+i e_{2c}$ in~$\Omega_J$, \eqref{Omega},
gives rise to weight spaces with weights having a coefficient~$-1$ for
$x_b$ and~$x_c$ and coefficient~$+1$ for the remaining~$x_a$.
Thus we define  three other real weight spaces with bases~$\{\mu_j,
\nu_j\}$ and weights as listed below:
\begin{equation}
\label{second}
\mu_2 =e^{1357}+e^{1467}-e^{1368}+e^{1458} {\rm \ and\ }
\nu_2=-e^{1468}-e^{1358}+e^{1457}-e^{1367}
\end{equation}
with weight\/~$2(\lambda_2-\lambda_4)=x_1+x_2-x_3-x_4$;

\begin{equation}
\label{third}
\mu_3=e^{1357}+e^{1467}+e^{1368}-e^{1458} {\rm \ and\ }
\nu_3=-e^{1468}-e^{1358}-e^{1457}+e^{1367}
\end{equation}
with weight\/~$2(\lambda_1-\lambda_2+\lambda_3)= x_1-x_2+x_3-x_4$; and

\begin{equation}
\label{fourth}
\mu_4=e^{1357}-e^{1467}+e^{1368}+e^{1458}{\rm \ and\ }
\nu_4=-e^{1468}+e^{1358}-e^{1457}-e^{1367}
\end{equation}
with weight\/~$2(\lambda_1-\lambda_3)= x_1-x_2-x_3+x_4$.
\end{lemma}

The decomposition in equation
\eqref{intrinsic} expresses~$\cay$ as a sum of a zero weight vector
and a highest weight vector for~$\pi_4$.

The intertwining diagram in Figure \ref{31} implies that~$\psi$ maps a
weight space of the representation~$\pitwo$ into the corresponding
weight space for the representation~$\pifour\circ\phi$.  Since~$\phi$
interchanges~$\lambda_1$ and~$\lambda_4$:
\begin{enumerate}
\item the weight space for
$2\lambda_1=2x_1$ in the representation~$\pifour\circ\phi$ is the weight
space for~$2\lambda_4= x_1+x_2+x_3+x_4$ in the representation
$\pifour$,
\item
the weight space for~$2(\lambda_2-\lambda_1)=2x_2$ in the
representation~$\pifour\circ\phi$ is the weight
space for~$2(\lambda_2-\lambda_4)= x_1+x_2-x_3-x_4$ in the representation
$\pifour$,
\item
the weight space for~$2(\lambda_4-\lambda_2+\lambda_3)=2x_3$ in the
representation~$\pifour\circ\phi$ is the weight space
for~$2(\lambda_1-\lambda_2+\lambda_3)= x_1-x_2+x_3-x_4$ in the
representation~$\pifour$
\item the weight space for~$2(\lambda_4-\lambda_3)=2x_4$ in the
representation~$\pifour\circ\phi$ is the weight
space for~$2(\lambda_1-\lambda_3)= x_1-x_2-x_3+x_4$ in the representation
$\pifour$.
\end{enumerate}
If we conjugate~$\phi$ by an element~$k$, and multiply~$\psi$ by~$\pifour(k)$ equation
\eqref{0} becomes
\begin{equation}\label{conjugate} \pifour(k\phi(g)k^{-1})\pifour(k)\psi(v)=\pifour(k)\psi(\pitwo(g)v).
\end{equation}
Conjugating by an appropriate element of the maximal torus, we can rotate the basis in each weight space and assume
\begin{equation}\label{1}
\psi(u_j)=\frac12 \mu_j,
\end{equation}
for~$j=1,\ldots,4$, and~$u_j$ is defined by \eqref{23}.  The
factor~$\frac12$ is required in order that~$\psi$ be an isometry.

Note that~$\pi_4(k)\psi (z)=\psi (z)$ for $k$ in the maximal torus and
$z$ a zero weight vector in $sigma^2_0(V)$.

The zero weight space of~$\sigma^2_0(V)$, when presented as
matrices, is the three dimensional space with an orthogonal basis
consisting of the matrices
$$
z_1=\left(\begin{array}{cc}
    I_4 & 0\\
    0 & -I_4
\end{array}\right),\quad z_2=
\left(\begin{array}{ccc}
    I_2& 0 & 0\\
    0 & -I_4 & 0\\
    0 &0 & I_2
\end{array}\right)$$
$$
z_3=
\left(\begin{array}{cccc}
    I_2 & 0 & 0 & 0\\
    0 & -I_2 & 0 & 0\\
    0 & 0 & I_2 & 0\\
    0 & 0 & 0 & -I_2
\end{array}\right).
$$

\begin{lemma}
The intertwiner~$\psi$ acts on the~$0$ weight space as follows:
\begin{equation}
\label{2}
\psi(z_1)=2 e^{1234},\quad \psi(z_2)=2e^{1278}, \quad
\psi(z_3)=2e^{1256}.
\end{equation}
\end{lemma}

\begin{proof}
The involution~$\phi$ leaves the simple root~$\alpha_3=x_3-x_4$
invariant, and hence also the real~$2$ dimensional subspace which is a
real form of the complex subspace of root vectors~$E_{\pm \alpha_3}$,
with a basis:
\begin{equation}
\label{33}
E_1=\left({\begin{array}{cccc}
    0 & 0& 0 & 0\\
    0 & 0& 0 & 0\\
    0 & 0& 0 & I_2\\
    0 & 0 & -I_2& 0
\end{array}}\right),
\quad E_2=
\left(\begin{array}{cccc}
    0 & 0& 0 & 0\\
    0 & 0& 0 & 0\\
    0 & 0& 0 & J\\
    0 & 0& J& 0
\end{array}\right).
\end{equation}

The element
\begin{equation*}
g_1=exp((\pi/2) E_1)\in \SO(8)
\end{equation*}
acting in~$\sigma^2_0(V)$ fixes~$z_1$ and interchanges~$z_2$ and
$z_3$, and acting in~$\Lambda^4_+$ it fixes~$e^{1234}$ and
interchanges~$e^{1256}$ and~$e^{1278}$. In fact, the element $g_1$
acts in the  coadjoint representation as reflection in $\alpha_3$. Since~$\phi(g_1)=g_1$, the image of~$z_1$ under~$\psi$ must be a multiple of~$e^{1234}$. The
isometry condition implies~$\psi(z_1)= \pm 2 e^{1234}$. We
normalize the multiple to~$+2$, using $-\psi$ if necessary and
another rotation, see equation~\eqref{conjugate}, by an element of 
the maximal torus to guarantee that $\psi(u_i)=\frac12 \mu_i$, \eqref{1}. 
The element
$g_2\in SO(8)$ acting in the coadjoint representation as Weyl reflection in $\alpha_2$ is also invariant under $\phi$. It interchanges $z_1$ and $z_3$ in 
the space of traceless symmetric matrices and interchanges~$e^{1234}$
and~$e^{1256}$ in the self-dual forms, so 
\begin{eqnarray*}
\psi(z_3)&=&\psi(\pi_2(g_2)z_1)\\
&=&\pi_4(\phi(g_2))\psi(z_1)\,\,\mbox{\rm eq. \eqref{0}}\\
&=&\pi_4(g_2)2e^{1234}\\
&=&e^{1256}
\end{eqnarray*}
A similar argument using the element whose coadjoint action is reflection in $\alpha_2+\alpha_3$ shows that $\psi(z_2)=2e^{1278}$ and completes the
proof of equation(\ref {2}).
\end{proof}

Putting together equations (\ref{1}) and (\ref{2}) define
\begin{equation}
\label{3}
A_1:=\left(
\begin{array}{cc} \frac 78& 0\\ 0 & -\frac18 I_7
\end{array}
\right),
\end{equation}
and
\begin{equation}
\label{4}
\cay:= e^{1234}+e^{1256}+e^{1278}+ e^{1357}-e^{1467}-e^{1368}- e^{1458}.
\end{equation}
Then
\begin{eqnarray}
\psi(A_1)&=& \psi\left(\frac18(z_1+z_2+z_3)
+\frac12(u_1)\right)\nonumber\\ &=&\frac14(e^{1234}+e^{1256}+e^{1278}+
e^{1357}-e^{1467}-e^{1368}-e^{1458})\nonumber\\ &=&\frac 14
\cay.\label{5}
\end{eqnarray}
\vskip5mm

\section{Proofs of Theorem~\ref{thm} and Theorem \ref{thm12} }
\label{five}

\begin{proposition}\label{comassone}
The self dual form
\begin{equation*}
\cay=e^{1234}+e^{1256}+e^{1278}+e^{1357}-e^{1467}-e^{1368}-e^{1458}
\end{equation*}
has comass~$1$.
\end{proposition}

\begin{proof}
We need to prove that
\begin{equation}
\sup_{g\in \SO(8)}(\cay, g (e_1\wedge e_2\wedge e_3\wedge e_4))=1
\end{equation}
First of all,~$\cay$ is self-dual and therefore orthogonal to the
anti-self dual part of~$e_1\wedge e_2\wedge e_3\wedge e_4$, so
we have
\begin{equation*}
(\cay, g (e_1\wedge e_2\wedge e_3\wedge e_4))=\frac 12(\cay, g
e^{1234}).
\end{equation*}
Next,
\begin{eqnarray*}
\frac 12 (\cay, g e^{1234})&=& \frac 14 (\cay, g\psi (z_1))\quad
\mbox{\rm by (\ref{2})}\\
&=&(\frac 14 \cay, g\psi(z_1))\\
&=&(\psi(A_1), g\psi(z_1))\quad\mbox {\rm by (\ref{5})}\\
&=&(\psi(A_1),\psi(\phi(g)z_1))\quad\mbox {\rm by (\ref{0})}\\
&=&(A_1, \phi(g)z_1),
\end{eqnarray*}
since~$\psi$ is an isometry.
Now
$$A_1=\left(
{\begin{array}{cc}
    1& 0\\
    0& 0_7
\end{array}}\right) - \frac 18 I_8,$$
and~$(I_8, \phi(g)z_1)=trace (\phi(g)z_1)=0$.  Putting this all
together we have
\begin{equation*}
\begin{aligned}
\sup_{g\in \SO(8)} &
\left( \cay, g (e_1\wedge e_2\wedge e_3\wedge e_4) \right)=
\\
&=
\sup_{g\in \SO(8)} \left(A_1+\frac 18 I_8, \phi(g)z_1\right)
\\
&=\sup_{g'=\phi(g)^{-1}\in \SO(8)} \left( g'\left( \begin{array}{cc} 1&
0
\\
0& 0_7
\end{array} \right) g^{\prime -1}, z_1 \right)
\\ &=\sup_{g'\in \SO(8)} \left\{\sum_{i=1,...,4}
(g'_{i1})^2-(g'_{i+4,1})^2 \right\} \\ &=1,
\end{aligned}
\end{equation*}
proving the result.
\end{proof}

\begin{proposition}\label{morecomassone}
The following self dual forms all have comass~$1$:
\begin{eqnarray*}
\omega_2&=&4\psi(\frac18(z_1+z_2+z_3)-\frac12 u_1)\\
&=&e^{1234}+e^{1256}+e^{1278}-e^{1357}+e^{1467}+e^{1368}+e^{1458}\\
\omega_3&=&4\psi(\frac18(z_1-z_2-z_3)+\frac12 u_2)\\
&=&e^{1234}-e^{1256}-e^{1278}+e^{1357}+e^{1467}-e^{1368}+e^{1458}\\
\omega_4&=&4\psi(\frac18(z_1-z_2-z_3)-\frac12 u_2)\\
&=&e^{1234}-e^{1256}-e^{1278}-e^{1357}-e^{1467}+e^{1368}-e^{1458}\\
\eta_1&=&4\psi(\frac18(z_1+z_2-z_3)-\frac12 u_3)\\
&=&e^{1234}-e^{1256}+e^{1278}+e^{1357}+e^{1467}+e^{1368}-e^{1458}\\
\eta_2&=&4\psi(\frac18(z_1+z_2-z_3)+\frac12 u_3)\\
&=&e^{1234}-e^{1256}+e^{1278}-e^{1357}-e^{1467}-e^{1368}+e^{1458}\\
\eta_3&=&4\psi(\frac18(z_1-z_2+z_3)-\frac12 u_4)\\
&=&e^{1234}+e^{1256}-e^{1278}+e^{1357}-e^{1467}+e^{1368}+e^{1458}\\
\eta_4&=&4\psi(\frac18(z_1-z_2+z_3)-\frac12 u_4)\\
&=&e^{1234}+e^{1256}-e^{1278}-e^{1357}+e^{1467}-e^{1368}-e^{1458}.
\end{eqnarray*}

\end{proposition}
\begin{proof}
Let~$D_i$ be the diagonal matrix with~$1$ the~$i$th position, all
other entries~$0$, and~$A_i:=D_i-\frac18 I$. The expressions in
parentheses on the right side of the equations above equal~$A_i$
for~$i=2,3,4$ and~$-A_i$ for~$i=5,6,7,8.$ These matrices are~$SO(8)$
conjugate, hence so are the corresponding self-dual~$4$-forms.
\end{proof}

For all the forms ~$\nu=\omega_j,$ or~$\nu=\eta_j,$~$j=1,2,3,4$ we have
\begin{equation}\label{max}
max_{g\in \SO(8)}(\nu, g (e_1\wedge e_2\wedge e_3\wedge e_4))=
(\nu,e_1\wedge e_2\wedge e_3\wedge e_4)=1.
\end{equation}
Any convex combination of the~$\omega_j,\eta_j$ also satisfies
\eqref{max} and therefore has comass~$1$. Conversely, any
self-dual~$4$-form satisfying \eqref{max} corresponds under triality
to a traceless symmetric~$8\times 8$ matrix~$A$ satisfying
 \begin{equation*}
max_{g\in \SO(8)}(A, gz_1)=1, \quad\makebox{\rm and }\quad (A, z_1)=1.
\end{equation*}
An elementary argument cf. \cite[Lemma 3.4]{DHM}, shows that any such
matrix is a convex combination of the matrices
$$\{A_1,\ldots, A_4,-A_5,\dots, -A_8\}.$$ Taking the image under
$\psi$, we see that any self-dual~$4$-form~$\nu$, satisfying
\eqref{max} is a convex combination of the~$\omega_j,\eta_j$. Any
comass~$1$ self-dual~$4$-form is~$\SO(8)$-conjugate to one satisfying
\eqref{max}, which we have just shown to be a convex combination of
the~$\omega_j,\eta_j$.

We will now prove Theorem~\ref{thm12}, to the effect that every
self-dual~$4$-form on~$\R^8$ satisfying \eqref{11}
is~$\SO(8)$-conjugate to the Cayley form.

\begin{proof}[Proof of Theorem~\ref{thm12}]
Let~$\omega$ be a self-dual~$4$-form satisfying \eqref{11}.  We
can assume that~$\omega$ is normalized to unit comass. As noted 
above, the comass~$1$ condition implies that~$\omega$ is conjugate 
to a convex combination 
\begin{equation}\label{pf} \omega=a_1 \omega_1 +
a_2\omega_2 + a_3 \omega_3 + a_4\omega_4 + a_5\eta_1 + a_6 \eta_2+
a_7 \eta_3+ a_8 \eta_4,
\end{equation}
with $a_i\geq 0$ and $\sum a_i=1$. Since the $\{\frac1{\sqrt{14}}\omega_i,  \frac1{\sqrt{14}}\eta_i\}$ form an orthonormal set, \eqref{pf} implies
$$\frac{|\omega|^2}{14}= \sum a_i^2\leq \sum a_i= 1,$$
with equality if and only if all the $a_i$ except one are zero.
Thus to achieve the maximum Euclidean norm 14, and satisfy
\eqref{max},  $\omega$ must be one of the~$8$
forms~$\{\omega_j,\eta_j|j=1,\ldots,4\}$ all of which
are~$\SO(8)$-conjugate to the Cayley form.
\end{proof}

\section{Stabilizer of the Cayley form}

In this section we give a proof using triality of Theorem~\ref{stabilizer}
stating that the stabilizer of the Cayley form is~$\Spin(7)$.
\begin{proof}
Recall, \eqref{0}, that there is a linear isometry ~$\psi:\sigma^2_0(V)\rightarrow \Lambda_+^4(V)$ such that
for all~$g\in Spin(8)$
$$ \pi_4(\phi(g)) \psi(v)=\psi(\pi_2(g)v),\quad{\mbox{\rm and}}\quad
\psi(A_1)=\cay,$$ where~$\phi$ be the triality automorphism
interchanging the fundamental weights~$\lambda_1$ and~$\lambda_4$
and~$A_1$ is the  diagonal matrix defined in the proof of
Proposition~\ref{morecomassone}. If~$G$ denotes the stabilizer of
$A_1$ in the representation~$\pi_2$, then the stabilizer of~$\cay$
in the representation~$\pi_4$ is~$\phi(G)$. Both representations
$\pi_2$ and~$\pi_4$ factor  through~$\rho_1:\Spin(8)\rightarrow
\SO(8)$. Composing with~$\rho_1$, we see that the stabilizer of
$\cay$ in the representation~$\hat\pi_4$ of~$\SO(8)$ (see
\eqref{factor}) is~$\rho_1\phi(G)$.

A simple matrix calculation shows
that the~$SO(8)$-stabilizer  of~$A_1\in \sigma^2_0(V)$ is the subgroup
${\rm O}(7)\cong \{\pm I_8\}\times \SO(7)\cong{\mathbb Z}_2\times \SO(7)$.
Let~$\gamma$ be the ``volume form" in the Clifford algebra:
 ~$$\gamma=f_1f_2f_3f_4f_5f_6f_7f_8,$$
which is also an element of~$\Spin(8)$.
In the vector representation~$\rho_1(\gamma)=-I_8;$
therefore,
$$ G =\rho_1^{-1}(\{\pm I_8\}\times \SO(7))= \{1, \gamma\}\times\rho_1^{-1}(\SO(7))=\{1, \gamma\}\times \Spin(7).$$
To complete the proof, we will show  that~$\rho_1\phi$ is
injective on the subgroup~$\Spin(7)$, that is,~${\rm
Ker}(\rho_1\phi)\cap\Spin(7)=\{1\}$.

The~$\pm 1$ eigenspaces of~$\gamma$ define the
splitting of the Clifford module:~$\Delta=\Delta_+\oplus \Delta_-$,
and the kernel of the representation~$\rho_4:\Spin(8)\rightarrow {\rm Aut}(\Delta_+)$
is~$\{1,\gamma\}$.

Since~$\phi$ conjugates the representation~$\rho_1$ to~$\rho_4$,
$$\phi(\{\pm 1\})=\phi(ker\rho_1)=ker\rho_4=\{1, \gamma\}.$$
In fact, triality induces a representation
of the symmetric group~$\Sigma_3$ on the center of~$\Spin(8)$,
$\{\pm 1,\pm \gamma\},$  permuting the non-identity elements
$\{-1,\gamma, -\gamma\}$, and~$\phi$
acts as the transposition of the first two elements.

Thus
$${\rm Ker}(\rho_1\phi)\cap \Spin(7)=\phi^{-1}({\rm Ker}\rho_1)\cap \Spin(7)$$
$$\hskip8mm=\phi({\rm Ker}\rho_1)\cap \Spin(7)=\{1,\gamma\}\cap \Spin(7)=\{1\}.~$$
This shows that~$\rho_1\phi$ is injective on~$\Spin(7)$, and completes the proof that the stabilizer of~$\cay$ in~$\SO(8)$ is isomorphic to~$\Spin(7)$.\end{proof}

The following classification by orbit type of comass 1 self-dual
$4$-forms (calibrating forms) is given in \cite{DHM}.
\begin{enumerate}
 \item
Type~$(1,0)$,~$\phi=\cay$, Cayley geometry;
 \item
Type~$(2,0)$,~$\phi=\frac12(\cay+\omega_2)=e^{1234}+e^{1256}+e^{1278}$,
K\"ahler 4-form, that is, the square of the K\"ahler form;
 \item
Type~$(3,0)$,
$\phi=\frac13(\cay+\omega_2+\omega_3)=\frac16(\tau_I^2+\tau_J^2+\tau_k^2)$,
~Kraines form, quaternionic geometry;
\item
Type~$(1,1)$,~$\phi=\frac12(\cay+\eta_4)=Re[(e_1+i e_7)(e_2 -i
e_8)(e_3+i e_5)(e_4-i e_6)]$, special Lagrangian geometry;
\item Type~$(2,1)$,~$\phi=\frac14 \cay+\frac12 \omega_2+\frac
14\eta_4$, ~$\mu=\frac14\cay +\frac14\cay+\frac12\eta_4$,
$\psi=\frac13(\cay+\omega_2+\eta_4)$, complex Lagrangian geometry;
\item
Type~$(2,2)$,
$\phi=\frac14(\cay+\omega_2+\eta_3+\eta_4)=(e_{12}+e_{78})(e_{34}+e_{56})$;
\item Type~$(3,1)$,~$\phi=\frac14(\cay+\omega_2+\omega_3+\eta_4)$;
\item Type~$(3,2)$,~$\psi=\frac15(\cay+\omega_2+\omega_3+\eta_3+\eta_4)$;
\item Type~$(3,3)$,~$\mu=\frac16(\cay
+\omega_2+\omega_3+\eta_2+\eta_3+\eta_4)$.
\end{enumerate}

\section{A counterexample}\label{counterx}
One might have thought that for any choice of coefficients~$\pm 1$ in a
linear combination of the forms
$$\{e^{1234},e^{1256},e^{1278},e^{1357},e^{1467},e^{1368},e^{1458}\}$$
would give a form
of comass~$1$, which would, therefore, realize the maximal Wirtinger ration
$14$.  However, a calculation similar to that in the proof of
Proposition~\ref{comassone} shows that the form~$\omega_+$ with all
coefficients~$+1$ has comass~$2$.

\begin{proposition}\label{comasstwo}
The self dual form
\begin{equation*}
\omega_+=e^{1234}+e^{1256}+e^{1278}+e^{1357}+e^{1467}+e^{1368}+e^{1458}
\end{equation*}
has comass~$2$.
\end{proposition}
\begin{proof}
We have~$\omega_+= \frac12 \omega_2-\frac12 \omega_4+\frac12
\eta_1+\frac12 \eta_3$.  Thus the corresponding symmetric traceless
matrix is
$$ A_+:=\frac12 (A_2)-\frac12 (A_4)+ \frac12(-A_5) +\frac12 (-A_7),$$ where
the~$A_j$ are defined above in the proof of Proposition~\ref{morecomassone};
that is,~$A_+=D_+ +\frac18 I$ where
$$D_+=\left(
{\begin{array}{cccccccc}
 0 & 0       & 0   & 0        & 0        & 0 & 0        & 0\\
 0 & \frac12 & 0   & 0        & 0        & 0 & 0        & 0\\
 0 & 0       & 0   & 0        & 0        & 0 & 0        & 0\\
 0 & 0       & 0   & -\frac12 & 0        & 0 & 0        & 0\\
 0 & 0       & 0   & 0        & -\frac12 & 0 & 0        & 0\\
 0 & 0       & 0   & 0        & 0        & 0 & 0        & 0\\
 0 & 0       & 0   & 0        & 0        & 0 & -\frac12 &0 \\
 0 & 0       & 0   & 0        & 0        & 0 & 0        & 0
\end{array}}\right).
$$ As in the proof of Proposition~\ref{comassone}
\begin{equation*}
\begin{aligned}
\sup_{g\in \SO(8)} &
\left( \omega_+ , g (e_1\wedge e_2\wedge e_3\wedge e_4) \right)=\\
&=
\sup_{g\in \SO(8)} \left(A_+ , \phi(g)z_1\right)\\
&=\sup_{g\in \SO(8)} \left(A_+ -\frac18 I, \phi(g)z_1\right)\\
&=\sup_{g'=\phi(g)\in \SO(8)} \left(D_+, g' z_1 g^{\prime -1}\right)\\
&\leq \frac12 (\sum_{j=2,4,5,7} \sup_{g'\in \SO(8)} \left(D_j ,
g' z_1 g^{\prime-1}\right)\\
&=2,
\end{aligned}
\end{equation*}
The last equality follows from the fact that the calculation of
\begin{equation*}
\sup_{g'\in \SO(8)} \left(D_1 , g' z_1 g^{\prime -1}\right)=1
\end{equation*}
in the proof of Proposition~\ref{comassone}, applies equally well to
the other ~$D_j$.  The value~$2$ for~$\left(D_+, g' z_1 g^{\prime
-1}\right)$ is actually achieved for the matrix
\begin{equation*}
g'=\left(
{\begin{array}{cccccccc}
 1 & 0 & 0 & 0  & 0 & 0 & 0 & 0\\
 0 & 1 & 0 & 0  & 0 & 0 & 0 & 0\\
 0 & 0 & -1 & 0  & 0 & 0 & 0 & 0\\
 0 & 0 & 0 & 0  & 0 & 0 & 0 & 1\\
 0 & 0 & 0 & 0  & 1 & 0 & 0 & 0\\
 0 & 0 & 0 & 0  & 0 & 1 & 0 & 0\\
 0 & 0 & 0 & 0  & 0 & 0 & 1 & 0 \\
 0 & 0 & 0 & 1  & 0 & 0 & 0 & 0
\end{array}}\right).
\end{equation*}
\end{proof}

A similar argument shows that the form
\noindent
\begin{eqnarray*}
\mu&=&e^{1234}-e^{1256}+e^{1278}+e^{1357}+e^{1467}+e^{1368}+e^{1458}\\
&=&\frac12
\omega_2 +\frac12 \omega_3+\frac12 \eta_1-\frac12 \eta_4
\end{eqnarray*}
has comass~$2$.

Note that, by considering the associated diagonal matrices, and the
action of the symmetric group,~$S_8$, it is clear that the comass~$1$
forms~$\frac12\omega_+= \frac14 \omega_2-\frac14 \omega_4+\frac14
\eta_1+\frac14 \eta_3$ and~$\frac12\mu=\frac14 \omega_2 +\frac14
\omega_3+\frac14 \eta_1-\frac14 \eta_4$ are~$\SO(8)$-conjugate,
respectively, to the convex combinations~$\frac14 \omega_2+\frac14
\eta_1+\frac14 \eta_2+\frac14 \eta_3$ and~$\frac14 \cay +\frac14
\omega_2+\frac14 \omega_3+\frac14 \eta_1$, both of orbit type~$(3,1)$
in the classification of \cite{DHM}.

\section{Acknowledgment}

We are grateful to the anonymous referee for a number of suggestions
that helped improve an earlier version of the manuscript.

\vfill\eject

\end{document}